\documentclass{amsart}

\usepackage{a4}
\usepackage{amsmath}
\usepackage{amsthm}
\usepackage{amssymb}
\usepackage{calc}
\usepackage{color}
\usepackage{delarray}
\usepackage{enumerate}
\usepackage{epsfig}
\usepackage{exscale}
\usepackage{graphicx}
\usepackage{latexsym}
\usepackage{multicol}
\usepackage{nextpage}
\usepackage{pifont}
\usepackage{rawfonts}
\usepackage{rotating}
\usepackage{url}
\usepackage{verbatim}
\usepackage{xspace}

\theoremstyle{plain}
  \newtheorem{theorem}{Theorem}[section]
  \newtheorem*{theorem*}{Theorem}
  \newtheorem{corollary}[theorem]{Corollary}
  \newtheorem*{corollary*}{Corollary}
  \newtheorem{lemma}[theorem]{Lemma}
  \newtheorem*{lemma*}{Lemma}
  
  \newtheorem*{proposition*}{Proposition}

\theoremstyle{definition}
  \newtheorem{algorithm}[theorem]{Algorithm}
  \newtheorem*{algorithm*}{Algorithm}
  
  \newtheorem*{assumption*}{Assumption}
  
  \newtheorem*{conjecture*}{Conjecture}
  
  \newtheorem*{definition*}{Definition}
  
  \newtheorem*{example*}{Example}
  
  \newtheorem*{hypothesis*}{Hypothesis}
  
  \newtheorem*{property*}{Property}
  \newtheorem{remark}[theorem]{Remark}
  \newtheorem*{remark*}{Remark}

\graphicspath{ {courbes/} }

\newcommand{\N}{\mathbb N}
\newcommand{\R}{\mathbb R}

\newcommand{\Sy}{\mathbb S}
\newcommand{\EE}{\mathbb E}
\newcommand{\M}{\mathcal M}
\newcommand{\T}{\mathcal T}
\newcommand{\bT}{\overline{\mathcal T}}
\newcommand{\U}{\mathcal U}
\newcommand{\F}{\mathcal F}
\newcommand{\K}{\mathcal K}
\newcommand{\Q}{\mathcal Q}
\newcommand{\G}{\mathcal G}
\newcommand{\W}{\mathcal W}

\newcommand{\DD}{\mathcal D}
\newcommand{\Po}{\mathcal P}
\newcommand{\Li}{\mathcal L}
\newcommand{\Ha}{\mathcal H}
\newcommand{\C}{\mathcal C}

\newcommand{\tr}{\mathop{\mathrm{tr}}}

\newcommand{\tp}{^{\mathsf T}\,}
\newcommand{\card}[1]{\# #1}

\newcommand{\A}{\mathfrak A}
\newcommand{\NXW}{N_{\mathrm{rg}}}
\newcommand{\NX}{N_x}
\newcommand{\NW}{N_w}
\newcommand{\Nzero}{N_{\mathrm{in}}}

\newcommand{\eps}{\epsilon}
\newcommand{\NEW}[1]{{\em #1\/}\index{#1}}
\newcommand{\Msmall}{\overline{\M}}
\newcommand{\cardMs}{\card{\Msmall}}

\def\<#1,#2>{\langle #1, #2\rangle}
\newcommand{\fu}{\underline{f}}
\newcommand{\gu}{g}
\newcommand{\sigmau}{\underline{\sigma}}

\newenvironment{skproof}{\begin{proof}[Sketch of proof]}{\end{proof}}
\newcommand{\bo}{_{\text{b}}}

\title[Probabilistic max-plus schemes 
for solving HJB equations]{Probabilistic max-plus schemes 
for solving Hamilton-Jacobi-Bellman equations}

\author{Marianne Akian \and Eric Fodjo}
\address{M.~Akian: INRIA and CMAP, \'Ecole Polytechnique CNRS. Address:
CMAP, \'Ecole Polytechnique,
Route de Saclay,
91128 Palaiseau Cedex}
\email[M.~Akian]{Marianne.Akian@inria.fr}
\address{E.~Fodjo: I-Fihn Consulting and 
INRIA and CMAP, \'Ecole polytechnique CNRS. Address:
CMAP, \'Ecole Polytechnique,
Route de Saclay,
91128 Palaiseau Cedex}
\email[E.~Fodjo]{eric.fodjo@polytechnique.edu}

\thanks{The first author was partially supported by the ANR project MALTHY, ANR-13-INSE-0003,
by ICODE, and by PGMO, a joint program of EDF and FMJH (Fondation Math{\'e}matique Jacques Hadamard)}



\begin{document}

\begin{abstract}

We consider fully nonlinear Hamilton-Jacobi-Bellman equations
associated to diffusion control problems involving 
a finite set-valued (or switching) control and possibly 
a continuum-valued control.
In previous works (Akian, Fodjo, 2016 and 2017), 
we introduced a lower complexity 
probabilistic numerical algorithm
for such equations by combining max-plus and numerical probabilistic approaches.
The max-plus approach is in the spirit of the one of
McEneaney, Kaise and Han  (2011), and is based on the distributivity
of monotone operators with respect to suprema.
The numerical  probabilistic approach is in the 
spirit of the one proposed by Fahim, Touzi and Warin (2011).
A difficulty of the latter algorithm was in the critical
constraints imposed on the Hamiltonian to ensure the monotonicity of the
scheme, hence the convergence of the algorithm.
Here, we present new probabilistic schemes which are monotone
under rather weak assumptions, and show error estimates for these schemes.
These estimates will be used in further works to
study the probabilistic max-plus method.

\end{abstract}

\maketitle

\section{Introduction}

\label{sec-int}

We consider a finite horizon diffusion control problem on $\R^d$ involving 
at the same time a ``discrete'' control taking its values in a finite set $\M$, 
and a ``continuum'' control taking its values in some subset $\U$ 
of a finite dimensional space $\R^p$ (for instance a 
convex set with nonempty interior), which we next describe.

Let $T$ be the horizon. 
The state $\xi_s\in\R^d$ at time $s\in [0,T]$
satisfies the stochastic differential equation
\begin{equation}\label{defxi}
d \xi_s = f^{\mu_s} (\xi_s, u_s) ds + \sigma^{\mu_s} (\xi_s, u_s) d W_s
\enspace , \end{equation}
where $(W_s)_{s\geq 0}$ is a $d$-dimensional
Brownian motion on a filtered probability space
$(\Omega,\F,(\F_s)_{0\leq s\leq T},P)$.
The control processes $\mu:=(\mu_s)_{0\leq s\leq T}$ and $u:=(u_s)_{0\leq s\leq T}$
take their values in the
sets $\M$ and $\U$ respectively and they are admissible if
they are  progressively measurable with respect to
the filtration $(\F_s)_{0\leq s\leq T}$.
We assume that, for all $m\in\M$,
the maps $f^m: \R^d\times \U\to \R^d$ and
$\sigma^m: \R^d\times \U\to \R^{d\times d}$ are continuous
and satisfy properties implying the existence of the process 
$(\xi_s)_{0\leq s\leq T}$ for any admissible control processes $\mu$ and $u$.

Given an initial time $t\in [0,T]$, the control problem consists in maximizing
the following payoff:
\begin{align*}
J(t, x, \mu, u) :=& \EE \left[ \int_t^T 
e^{ - \int_t^s \delta^{\mu_{\tau}} (\xi_{\tau}, u_{\tau}) d{\tau} }
\ell^{\mu_s} (\xi_s, u_s) ds \right.\\ & \left. \qquad 
+ e^{ - \int_t^T \delta^{\mu_{\tau}} (\xi_{\tau}, u_{\tau}) d{\tau} }
\psi(\xi_T)  \mid \xi_t = x \right] \enspace ,
\end{align*}
where, for all $m\in \M$, $\ell^m: \R^d\times \U\to \R$,
$\delta^m: \R^d\times \U\to [0,+\infty)$,
and $\psi:\R^d\to\R$ are  given continuous maps.
We then define the value function of the problem as the optimal payoff:
$$v(t, x) = \sup_{\mu, u} J(t,x,\mu, u)\enspace ,$$
where the maximization holds over all admissible control processes
$\mu$ and $u$.

Let $\Sy_d$ denotes the set of symmetric $d\times d$ matrices
and let us denote by $\leq$ the Loewner order on $\Sy_d$
($A\leq B$ if $B-A$ is nonnegative).
The Hamiltonian $\Ha:\R^d\times \R \times \R^d\times \Sy_d\to \R$
of the above control problem is defined as:
\begin{subequations}\label{defH}
\begin{align}
\Ha(x, r,p,\Gamma):=&\max_{m\in \M} \Ha^m(x, r,p,\Gamma)\enspace,
\label{defHmax}\end{align}
with
\begin{align}
 \Ha^m(x, r,p,\Gamma):=&\max_{u\in \U}  \Ha^{m,u}(x, r,p,\Gamma) 
\enspace ,\label{defHm}\\
 \Ha^{m,u}(x, r,p,\Gamma):=&
\frac{1}{2} \tr\left(\sigma^m(x,u) \sigma^m(x,u)\tp \Gamma
\right)  +f^m(x,u)\cdot p\nonumber\\
&\qquad -\delta^m(x,u)r +\ell^m(x,u) \enspace.
\label{defHmu}
\end{align}
\end{subequations}

Under suitable assumptions, the value function 
$v:[0,T]\times \R^d\to \R$ is the unique 
(continuous) viscosity solution of
the following Hamilton-Jacobi-Bellman equation
\begin{subequations}\label{HJB}
\begin{align}
&-\frac{\partial v}{\partial t} 
-\Ha(x, v(t,x), Dv(t,x), D^2v(t,x))=0, 
\quad  x\in \R^d,\; t\in [0,T), \label{HJB1}\\
& v(T,x)=\psi(x), \quad x\in \R^d,
\end{align}
\end{subequations}
satisfying also some growth condition at infinity (in space).

In~\cite{touzi2011}, Fahim, Touzi and Warin proposed a probabilistic numerical 
method to solve such fully nonlinear partial 
differential equations~\eqref{HJB}, inspired by their backward stochastic
differential equation interpretation 
given by Cheridito, Soner, Touzi and Victoir
in~\cite{cheridito2007}.
This method consists in two steps, the first one 
beeing a time discretization of the partial differential
equation using the Euler discretization of the stochastic differential
equation of an uncontrolled diffusion (thus different from the controlled one).
The second step of the method is based on the simulation of
the discretized diffusion and linear regression estimations which
can be seen as an alternative to a space discretization.

In~\cite{mceneaney2010,mceneaney2011}, McEneaney, Kaise and Han
proposed an idempotent numerical method which works at least when 
the Hamiltonian with fixed discrete control, $\Ha^m$,
correspond to linear quadratic control problems.
This method is based on the distributivity of the (usual) addition operation
over the supremum (or infimum) operation, and on a property
of invariance of the set of quadratic forms.
It computes in a backward manner the value function $v(t,\cdot)$
at time $t$ as a supremum of quadratic forms.
However, as $t$ decreases, the number of quadratic forms generated 
by the method increases exponentially (and even become infinite
if the Brownian is not discretized in space) and some pruning 
is necessary to reduce the complexity of the algorithm.

In~\cite{fodjo1}, we introduced an algorithm
combining  the  two above methods at least in their spirit.
The algorithm applies the first step (the time discretization) of the
method of~\cite{touzi2011} to the HJB equations obtained when the
discrete control is fixed, then using the simulation of as many
uncontrolled stochastic processes as  discrete controls,
it applies a max-plus type space discretization in the spirit of
the method of~\cite{mceneaney2010,mceneaney2011}.
Then, without any pruning, 
the number of quadratic forms representing the value function is bounded
by the sampling size~\cite{fodjo1}.
Hence, the complexity of the algorithm
is bounded polynomially in the number of discretization time steps
and the sampling size.

The convergence of the probabilistic max-plus algorithm proposed
in~\cite{fodjo1} is based, as for the one of~\cite{touzi2011},
on the monotonicity of the time discretization scheme.
In particular~\cite{touzi2011}, this monotonicity allows one to apply 
the theorem of Barles and Souganidis~\cite{barles90}.
However, for this monotonicity to hold, critical
constraints are imposed on the Hamiltonian: 
the diffusion matrices $\sigma^m(x,u) \sigma^m(x,u)\tp$ need
at the same time to be bounded from below (with respect to the Loewner order)
by a symmetric positive definite matrix $a$ and 
bounded from above by $(1+2/d) a$.
Such a constraint is restrictive, in particular it may not hold
even when the matrices  $\sigma^m(x,u)$ do not depend on $x$ and $u$ but take
different values for $m\in \M$. %
In~\cite{monotone-zhang2015}, Guo, Zhang and Zhuo proposed a monotone
scheme exploiting the diagonal part of the diffusion matrices 
and combining a usual finite
difference scheme to the scheme of~\cite{touzi2011}.
This scheme can be applied in more general situations than the one 
of~\cite{touzi2011}, but still does not work for general 
control problems.
In~\cite{fodjo2}, we proposed a new probabilistic discretization scheme 
of the second order derivatives which allowed us to obtain
the monotonicity of the time discretization of HJB equations~\eqref{HJB}
with bounded coefficients and an ellipticity condition. Indeed,
the monotonicity holds when the first
order terms of the HJB equation are dominated by the second order ones.

Here, we  propose a new probabilistic scheme for the first order derivatives
which is in the spirit of the upwind discretizations used by Kushner for
optimal control problems, see for instance~\cite{kushner2}.
This allows one to solve also degenerate equations or to use 
time discretizations based on the simulation of a diffusion with same 
variance as the controlled process.

As soon as the convergence of the algorithm holds, 
one may expect to obtain estimates on the error leading to 
bounds on the complexity as a function of the error.
Both depend on the error of the time discretization on
the one hand, and the error of the ``space discretization'' on the other hand.
We shall only study here the error of the time discretization,
for which we obtain error estimates similar
to the ones in~\cite{touzi2011}, using the results of
Barles and Jakobsen~\cite{jakobsen07}.
We shall also show how to adapt the method of~\cite{fodjo1,fodjo2}
with the new time discretization scheme.

The paper is organized as follows.
In Section~\ref{sec-touzi}, we recall the scheme of~\cite{touzi2011}.
Then, monotone probabilistic discretizations of second order and first 
order derivatives are presented in Section~\ref{sec-derivatives},
with error estimates for regular functions. 
These discretizations and error estimates 
are applied to Hamilton-Jacobi-equations 
in Section~\ref{sec-monotone}, for which the error on a bounded
Lipschitz solution is obtained by using the results of 
Barles and Jakobsen~\cite{jakobsen07}. 
In Section~\ref{sec-maxplus}, we recall the algorithm of~\cite{fodjo1,fodjo2}
and show how it can be combined with the scheme of
Section~\ref{sec-monotone}.

\section{The probabilistic time discretization of Fahim, Touzi and Warin}
\label{sec-touzi}

Let us first recall the first step of the 
probabilistic numerical scheme proposed by
Fahim, Touzi and Warin in~\cite{touzi2011},
which can be viewed as a time dicretization.

Let $h$ be a time discretization step such that $T/h$ is an integer.
We denote by $\T_h=\{0,h,2h,\ldots, T-h\}$ and $\bT_h=\{0,h,2h,\ldots, T\}$
the set of discretization times
of $[0,T)$ and $[0,T]$ respectively.
Let $\Ha$ be any hamiltonian of the form~\eqref{defH}.
Let us decompose $\Ha$ as the sum of the (linear) generator 
$\Li$ of a diffusion (with no control) and of
a nonlinear elliptic Hamiltonian $\G$, that is
$\Ha=\Li+\G$ with 
\begin{align*}
 \Li(x, ,r, p,\Gamma)= \Li(x, p,\Gamma):=&
\frac{1}{2} \tr\left(a(x)  \Gamma \right) +\underline{f}(x)\cdot p \enspace,
\end{align*}
$a(x)=\underline{\sigma}(x) \underline{\sigma}(x)\tp$ 
and $\G$ such that $a(x)$ is positive definite and
$\partial_{\Gamma} {\G}$ is positive semidefinite, for all $x\in\R^d,\;
r\in\R, p\in \R^d, \Gamma\in \Sy_d$.
Denote by $\hat{X}$
the Euler discretization of the diffusion with generator $\Li$:
\begin{equation}
\hat{X}(t+h)=
 \hat{X}(t)+ \underline{f}(\hat{X}(t))h+
\underline{\sigma}(\hat{X}(t)) (W_{t+h} -W_t)\enspace .
\label{xhat} \end{equation}
The time discretization of \eqref{HJB} proposed in~\cite{touzi2011} 
has the following form:
\begin{equation}\label{scheme}
 v^h(t,x)=T_{t,h}(v^h(t+h,\cdot))(x),\quad t\in\T_h\enspace,\end{equation}
with %
\begin{equation}
T_{t,h}(\phi)(x)= %
 \DD_{t,h}^{0}(\phi)(x)
+ h \G(x, \DD_{t,h}^{0}(\phi)(x),\DD_{t,h}^{1}(\phi)(x),\DD_{t,h}^{2}(\phi)(x))
\enspace ,
\label{def-th-m}
\end{equation} %
where, for $i=0,1,2$, 
$ \DD_{t,h}^i(\phi)$ is the approximation of the $i$th differential
of $e^{h\Li}\phi$
obtained using the following scheme:
\begin{subequations}\label{defDcal}
\begin{align}
 \DD_{t,h}^i(\phi)(x)&=\EE(D^i\phi(\hat{X}(t+h))\mid \hat{X}(t)=x)\\
&=\EE(\phi(\hat{X}(t+h))
\Po^i_{t,x,h}(W_{t+h}-W_t) \mid \hat{X}(t)=x)\label{defDcal2}
\enspace ,
\end{align}
\end{subequations}
where, $D^i$ denotes the $i$th differential operator,
and for all $t,x,h,i$, $\Po^i_{t,x,h}$ is
the polynomial of degree $i$  in the variable $w\in \R^d$
given by:
\begin{subequations}\label{defpol}
\begin{align}
 \Po^0_{t,x,h}(w)&=1\enspace,\\ \Po^1_{t,x,h}(w)&=(\underline{\sigma}(x)\tp)^{-1}h^{-1} w\enspace, \label{defpol2}\\
 \Po^2_{t,x,h}(w)&=(\underline{\sigma}(x)\tp)^{-1} h^{-2}(w w\tp-h I)
(\underline{\sigma}(x))^{-1}  \label{defpol3}
\enspace ,
\end{align}
\end{subequations}
where $I$ is the $d\times d$ identity matrix.
Note that the second equality in~\eqref{defDcal} holds 
for all $\phi$ with exponential
growth~\cite[Lemma 2.1]{touzi2011}.

In~\cite{touzi2011}, the convergence of the time discretization 
scheme~\eqref{scheme} is 
proved by using the theorem of Barles and Souganidis of~\cite{barles90},
under the above assumptions together with the critical
assumption that
$\partial_{\Gamma} {\G}$ is lower bounded by
some positive definite matrix (for all $x\in\R^d,\;
r\in\R, p\in \R^d, \Gamma\in \Sy_d$)
and that \sloppy
$\tr(a(x)^{-1} \partial_{\Gamma} {\G})\leq 1$.

Indeed, let us say that an operator $T$ between any 
partially ordered sets $\F$ and $\F'$
of real valued functions (for instance the set of bounded functions 
from some set $\Omega$ to $\R$, or $\R^n$)
is \NEW{$L$-almost monotone}, for some constant $L\geq 0$, if 
\begin{equation}\label{amonotone}
\phi,\psi\in \F,\; 
 \phi\leq \psi \implies T(\phi)\leq T(\psi)+
L \sup(\psi-\phi)\enspace ,\end{equation}
and that it is \NEW{monotone}, when this holds for $L=0$.

The above conditions together with the boundedness of 
$\partial_{p} {\G}$ are used to show (in Lemma 3.12 and 3.14
of~\cite{touzi2011}) that the operator
 $T_{t,h}$ is a $Ch$-almost monotone operator
over the set of Lipschitz continuous functions from $\R^d$ to $\R$.
Then, this property, together with other technical assumptions, are used
to obtain the assumptions of the theorem of Barles and Souganidis 
of~\cite{barles90}, and also estimates in the same spirit as in~\cite{jakobsen07}.

In~\cite{fodjo1}, we proposed to bypass the critical constraint, by
assuming that the Hamiltonians $\Ha^m$ (but not necessarily $\Ha$)
satisfy the critical constraint, 
and applying the above scheme to the Hamiltonians $\Ha^m$.

In~\cite{fodjo2}, we proposed an approximation of 
$\EE(D^2\phi(\hat{X}(t+h))\mid \hat{X}(t)=x)$ or $D^2 \phi(x)$
that we recall in the next section. It is
expressed as a conditional expectation as in~\eqref{defDcal2} but
depend on the derivatives of $\G$ with respect to $\Gamma$
at the given point,
via the matrices $\sigma^m(x,u)$ of the control problem.
Below, we also propose an approximation of 
$\EE(D\phi(\hat{X}(t+h))\mid \hat{X}(t)=x)$ or $D \phi(x)$ 
which is monotone in itself and thus allows one to consider the case where
the derivatives of $\G$ with respect to $\Gamma$ are zero or degenerate 
nonnegative matrices.

\section{Monotone probabilistic approximation of first and second order derivatives and their estimates}
\label{sec-derivatives}
We first describe the approximation of the second order derivatives 
proposed in~\cite{fodjo2}. 
Consider any matrix $\Sigma\in \R^{d\times \ell}$ 
with $\ell\in\N$ and let us denote by $\Sigma_{.j},\; j=1,\ldots \ell$,
its columns.
We denote by $\C^{k}([0,T]\times \R^d)$ or simply $\C^{k}$
the set of functions from $[0,T]\times\R^d$ to $\R$ 
with continuous partial derivatives up to order $k$ in $t$ and $x$,
and by $\C^{k}\bo ([0,T]\times \R^d)$ or $\C^{k}\bo$ the subset of functions with bounded
such derivatives.
Then, for any $v\in \C^2$, we have 
\begin{eqnarray}
\frac{1}{2}\tr( \underline{\sigma}(x)\Sigma\Sigma\tp\underline{\sigma}\tp(x)\, D^2v(t,x))&=&
\frac{1}{2}\sum_{j=1}^\ell\Sigma_{\cdot j}\tp\underline{\sigma}\tp(x)\, D^2v(t,x)\underline{\sigma}(x)\Sigma_{\cdot j}\enspace .\label{second-order}
\end{eqnarray}
For any integer $k$, consider the polynomial:
\begin{subequations}\label{defpoly2k}
\begin{gather}
 \Po^2_{\Sigma,k}(w)= \sum_{j=1}^{\ell}\|\Sigma_{.j}\|_2^{2}
\left(c_k \left(\frac{[\Sigma\tp w]_{j}}{\|\Sigma_{.j}\|_2}\right)^{4k+2}
-d_k\right)\enspace ,
\end{gather} with
\begin{gather}
 c_k:= \frac{1}{(4k+2)\EE\left[N^{4k+2}\right]}\enspace ,\quad
d_k:=\frac{1}{4k+2}\enspace ,
\end{gather}
\end{subequations}
where $N$ is a one dimensional normal random variable,
and where we use the convention that the $j$th term of the sum is zero when
$\|\Sigma_{.j}\|_2=0$.
This is the sum of the same expression defined for each column 
$\Sigma_{.j}$ instead of $\Sigma$.

Let $v\in\C^{4}\bo$, and $\hat{X}$ as in~\eqref{xhat},
then, the following expression is an approximation of~\eqref{second-order}
with an error in $O(h)$ uniform in $t$ and $x$~\cite[Th.\ 1]{fodjo2}:
\begin{gather}
h^{-1}\EE\left[v(t+h,\hat{X}(t+h)) \Po^2_{\Sigma,k}(h^{-1/2}(W_{t+h}-W_t))
 \mid \hat{X}(t)=x\right] \enspace . \label{approx-second-order}
\end{gather}

In order to obtain error estimates, we need the more precise following result.
For $p$ and $q$ two integers and $\phi$ a function from  $[0,T] \times \R^d$
to $\R$ with partial derivatives up to order $p$ in $t$ and  $q$ in $x$,
 we introduce the following notation :
\begin{align*}
|\partial^{p}_t D^{q} \phi| = \underset{\substack{
{(t,x) \in [0,T] \times \R^d} \\ 
{(\beta_i)_i \in \N^d, \sum_i \beta_i = q}}}{\sup} \left|\frac{\partial^{i+q} \phi}{\partial t^p \partial x_1^{\beta_1} \ldots \partial x_d^{\beta_d}}(t,x)\right|
\end{align*}
In the sequel, $\|\cdot\|$ will denote any norm on $\R^{d}$ or
on  $\R^{d\times d}$. Also $[x]_i$ will denote the $i$th coordinate 
of any vector $x\in\R^d$, and $[A]_{ij}$ will denote the $(i,j)$ entry of 
any matrix $A\in \R^{d\times \ell}$.

\begin{theorem}
\label{theo-er3}
Let $\hat{X}$ as in~\eqref{xhat}, and denote $W^t_h=W_{t+h}-W_t$.
Consider any matrix 
$\Sigma\in \R^{d\times \ell}$ with $\ell\leq d$.
Assume that $\fu$ and $\sigmau$ are bounded by some constant $C$ uniformely
in ($t$ and) $x$, and let $M$ be an upper bound of
$\|\Sigma \Sigma\tp\|$.
Then, there exists $K=K(C,M) > 0$ 
such that, for all $v \in \C^4\bo ([0,T]\times \R^d)$, we have,
for all $(t,x)\in \T_h\times \R^d$,
\begin{align*}
&\Bigg| h^{-1}\EE\left[v(t+h,\hat{X}(t+h)) \Po^2_{\Sigma,k}(h^{-1/2}W^t_h)
 \mid \hat{X}(t)=x\right]\nonumber  \\
&\quad -\frac{1}{2}\tr( \underline{\sigma}(x)\Sigma\Sigma\tp\underline{\sigma}\tp(x)\, D^2v(t,x)) \Bigg| \\
&\le K (1+\sqrt{h})^4
\Big[h (|\partial^{1}_t D^{2} v| + |\partial^{0}_t D^{3} v| + |\partial^{0}_t D^{4} v|) + \\
&\quad h\sqrt{h} |\partial^{1}_t D^{3} v| + h^2 |\partial^{2}_t D^{2} v| + h^2\sqrt{h} |\partial^{3}_t D^{1} v| + h^3 |\partial^{4}_t D^{0} v| \Big]\enspace . 
\end{align*}
\end{theorem}
\begin{skproof}
The proof follows from the following lemma and the property that
$h^{-1/2}W^t_{h}$ is a normal random vector and that normal
random variables have all their moments finite. 
\end{skproof}

\begin{lemma}
\label{theo-er2}
Let $v$, $W^t_h$ and $\Sigma$ be as in Theorem~\ref{theo-er3}.
We have, for all $(t,x)\in \T_h\times \R^d$,
\begin{align*}
&h^{-1}\EE\left[v(t+h,\hat{X}(t+h)) \Po^2_{\Sigma,k}(h^{-1/2}W^t_h)
 \mid \hat{X}(t)=x\right] \\
&=\frac{1}{2}\tr( \underline{\sigma}(x)\Sigma\Sigma\tp\underline{\sigma}\tp(x)\, D^2v(t,x)) + \frac{h}{2}\tr( \underline{\sigma}(x)\Sigma\Sigma\tp\underline{\sigma}\tp(x)\, 
\frac{\partial D^2v(t,x)}{\partial t}) \\
&\quad +\frac{h}{2}\sum_{i,j,p}\left(
\frac{\partial^3v}{\partial x_i \partial x_j \partial x_p}(t,x)
[\underline{\sigma}(x)\Sigma\Sigma\tp\underline{\sigma}\tp(x)]_{ij}
[\fu(x)]_p\right) \\
&\quad + \EE\left[M^4(v,t,x,W^t_h)\Po^2_{\Sigma,k}(h^{-1/2}W^t_h )\right]\enspace , 
\end{align*}
where there exists $(s,\xi)$ (random) equal to a convex combination of 
$(t,x)$ and $(t+h, x+\fu(x) h+\sigmau(x)W^t_h)$ such that:
\begin{align*}
M^4(v,t,x,W^t_h) = &\frac{h^3}{24} \frac{\partial^4 v}{\partial t^4}(s,\xi) \\
                           & + \frac{h^2}{6} \sum_{i=1}^d \frac{\partial^4 v}{\partial t^3\partial x_i}(s,\xi)[\fu(x)h+\sigmau(x) W^t_h]_i \\
                           & +  \frac{h}{4} \sum_{i,j} \frac{\partial^4 v}{\partial t^2\partial x_i \partial x_j}(s,\xi)[\fu(x)h+\sigmau(x) W^t_h]_i [\fu(x)h+\sigmau(x) W^t_h]_j\\
                           & + \frac{1}{6} \sum_{i,j,p} \frac{\partial^4 v}{\partial t\partial x_i \partial x_j \partial x_p}(s,\xi)[\fu(x)h+\sigmau(x) W^t_h]_i [\fu(x)h+\sigmau(x) W^t_h]_j \\
                           &\qquad [\fu(x)h+\sigmau(x) W^t_h]_p\\
                           & + \frac{1}{24h} \sum_{i,j,p,q} \frac{\partial^4 v}{\partial x_i \partial x_j \partial x_p \partial x_q}(s,\xi)[\fu(x)h+\sigmau(x) W^t_h]_i [\fu(x)h+\sigmau(x) W^t_h]_j \\
                           &\qquad [\fu(x)h+\sigmau(x) W^t_h]_p [\fu(x)h+\sigmau(x) W^t_h]_q\enspace .
\end{align*}
\end{lemma}
\begin{skproof}
Apply a Taylor expansion of $v$ around $(t,x)$ and use the property
that for each $j$, there exists a unitary matrix $U$ with 
$j$th column equal to $ \Sigma_{.j}/\|\Sigma_{.j}\|_2$, 
so that
$U\tp (h^{-1/2}W^t_{h})$ is a $d$-dimensional normal random vector
with $j$th coordinate equal to 
$[\Sigma\tp (h^{-1/2}W^t_{h})]_j/\|\Sigma_{.j}\|_2$. 
\end{skproof}

Let us also introduce the following approximation 
of the first order derivatives. For any vector $\gu\in \R^{d}$, 
consider the piecewise linear function $\Po^1_{\gu}$ on $\R^d$ :
\begin{align}\label{poly1}
\Po^1_{\gu}(w) =& 2( \gu_+ \cdot w_+ + \gu_- \cdot w_- )\enspace ,
\end{align}
where for any vector $\mu \in \R^d$, 
 $\mu_+, \mu_- \in \R^d$ are defined such that $[\mu_+]_i = \max([\mu]_i,0)$, 
$[\mu_-]_i = -\min([\mu]_i,0)$. Note that $\Po^1_{\gu}$ is
nonnegative. We shall show that
\begin{equation}\label{discret-ordre1-der}
\EE\Big[(v(t+h, \hat{X}(t+h))-v(t,x))\Po^1_{\gu}(h^{-1}W^t_h) \Big]\end{equation}
is a monotone approximation of
\[ (\sigmau(x)\gu) \cdot Dv (x)\enspace .\]
Before this, let us 
note that if $\sigmau(x)=1$, $\fu(x)=1$ and $h^{-1/2}W^t_h$ is discretized
by a random variable taking the values $1$ and $-1$ with probability $1/2$,
then the discretization $\DD_{t,h}^i(v(t+h,\cdot))(x)$
defined in~\eqref{defDcal2} is equivalent to 
a centered discretization of $Dv(x)$ with space step $\Delta x=h^{1/2}$,
whereas~\eqref{discret-ordre1-der} corresponds to the Kushner (upwind)
discretization~\cite{kushner2}
\[\sum_{i=1}^d \left[ [\gu_i]_+  \frac{v(t+h,x+h^{1/2} e_i)-v(t,x)}{h^{1/2}}
+[\gu_i]_-  \frac{v(t+h,x-h^{1/2} e_i)-v(t,x)}{h^{1/2}} \right]\enspace .\]

Using the same proof arguments as above we obtain the following results,
where Theorem~\ref{theo-firstorder} uses Lemma~\ref{lemma-firstorder}.

\begin{theorem}\label{theo-firstorder}
Let $\hat{X}$ as in~\eqref{xhat}, and denote $W^t_h=W_{t+h}-W_t$.
Consider any vector $\gu\in \R^{d}$.
Assume that $\fu$ and $\sigmau$ are bounded by some constant $C$ uniformely
in ($t$ and) $x$, and let $M$ be an upper bound of $\|\gu\|$.
Then, there exists $K=K(C,M) > 0$ 
such that, for all $v \in \C^2\bo ([0,T]\times \R^d)$, we have,
for all $(t,x)\in \T_h\times \R^d$,
\begin{gather*}
\left|(\sigmau(x)\gu) \cdot Dv- \EE\Big[(v(t+h, \hat{X}(t+h))-v(t,x))\Po^1_{\gu}(h^{-1}W^t_h) \Big]\right| \\
\le K (1+\sqrt{h})^2 \Big[\sqrt{h}(|\partial^{1}_t D^{0} v| + |\partial^{0}_t D^{1} v| + |\partial^{0}_t D^{2} v|) \\
+ h(|\partial^{1}_t D^{1} v|) + h\sqrt{h} |\partial^{2}_t D^{0} v| \Big]\enspace .
\end{gather*} 
\end{theorem}

\begin{lemma}\label{lemma-firstorder} 
Let $v$, $W^t_h$ and $\gu$ be as in Theorem~\ref{theo-firstorder}.
For all $(t,x)\in \T_h\times \R^d$,
there exists $(s,\xi)$ (random) equal to a convex combination of 
$(t,x)$ and $(t+h, x+\fu(x) h+\sigmau(x)W^t_h)$ such that:
\begin{align*}
(\sigmau(x)\gu) \cdot Dv &= 2\EE\Big[(v(t+h, \hat{X}(t+h))-v(t,x))\Po^1_{\gu}(h^{-1}W^t_h) \Big] \\
 & - 2 h (\frac{\partial v}{\partial t}(t,x) + \fu(x) \cdot Dv(t,x)) 
\EE[\Po^1_{\gu}(h^{-1}W^t_h)]\\
 &-h^2\EE[\frac{\partial^2 v}{\partial t^2}(s,\xi) \Po^1_{\gu}(h^{-1}W^t_h)]\\
 & - 2h\EE[(\fu(x)h + \sigmau(x)W^t_h)\cdot \frac{\partial}{\partial t}Dv(s,\xi) \Po^1_{\gu}(h^{-1}W^t_h)] \\
 & -\EE[(\fu(x)h + \sigmau(x)W^t_h)^{\intercal} D^2v(s,\xi) (\fu(x)h + \sigmau(x)W^t_h)\Po^1_{\gu}(h^{-1}W^t_h)]\enspace .
\end{align*}
\end{lemma}

We shall also need the following bound, that can be proved along the same
lines as the previous theorems. We do not give the proof since it can 
be bypassed by using alternatively the proof of Lemma 3.22 in~\cite{touzi2011}.
\begin{lemma}\label{lemma-zeroorder} 
Let $\Li$,  $\hat{X}$ and  $\DD_{t,h}^0$  be as in Section~\ref{sec-touzi}.
Denote $W^t_h=W_{t+h}-W_t$.
Assume that $\fu$ and $\sigmau$ are bounded by some constant $C$ uniformely
in ($t$ and) $x$. Then, there exists $K=K(C) > 0$ 
such that, for all $v \in \C^4\bo ([0,T]\times \R^d)$, we have,
for all $(t,x)\in \T_h\times \R^d$,
\begin{align*}
&\left| h^{-1}(\DD_{t,h}^0(v(t+h,\cdot))-v(t,x))-(\partial^1_t v+\Li(x,Dv(t,x),D^2v(t,x)))\right|=\\
&\left|h^{-1}(\EE(v(t+h,\hat{X}(t+h))\mid \hat{X}(t)=x)-v(t,x))-(\partial^1_t v+\Li(x,Dv(t,x),D^2v(t,x)))\right|\\
&\le K (1+\sqrt{h})^4
\Big[ h (|\partial^{0}_t D^{2} v| + |\partial^{1}_t D^{1} v| + |\partial^{2}_t D^{0} v| + |\partial^{0}_t D^{3} v| + |\partial^{1}_t D^{2} v| + |\partial^{0}_t D^{4} v|) \\
&\quad + h\sqrt{h} |\partial^{1}_t D^{3} v| + h^2 (|\partial^{2}_t D^{2} v| + |\partial^{2}_t D^{1} v|) + h^2\sqrt{h} |\partial^{3}_t D^{1} v| + h^3 |\partial^{4}_t D^{0} v| \Big]\enspace . 
\end{align*}
\end{lemma}

\section{Monotone probabilistic schemes for HJB equations}
\label{sec-monotone}

We shall apply the above approximations
of the first and second order derivatives in~\eqref{HJB}
in the same way as in~\cite{fodjo2}.
Let us decompose the hamiltonian $\Ha^{m,u}$
of~\eqref{defHmu} as $\Ha^{m,u}=\Li^{m}+\G^{m,u}$ with 
\begin{align*}
 \Li^{m}(x, p,\Gamma):=&
\frac{1}{2} \tr\left(a^{m}(x)  \Gamma \right) +\underline{f}^{m}(x)\cdot p \enspace,
\end{align*}
and $a^{m}(x)=\underline{\sigma}^{m}(x) \underline{\sigma}^{m}(x)\tp$,
and denote by $\hat{X}^m$
the Euler discretization of the diffusion with generator $\Li^m$.
We may choose the same linear operator $\Li^{m}$ for different
values of $m$, which is the case in Algorithm~\ref{algo1} below.
Assume that $a^{m}(x)$ is positive definite and that
$a^{m}(x)\leq \sigma^m(x,u) \sigma^m(x,u)\tp $ for all 
$x\in\R^d,\; u\in \U$,
and denote by $\Sigma^m(x,u)$ any $d\times \ell$ matrix such that
\begin{equation}\label{assump1}
\sigma^m(x,u) \sigma^m(x,u)\tp -a^{m}(x)= \underline{\sigma}^{m}(x) \Sigma^m(x,u)\Sigma^m(x,u)\tp\underline{\sigma}^{m}(x)\tp
\enspace .
\end{equation}
One may use for instance a Cholesky factorization of the matrix
$\underline{\sigma}^{m}(x)^{-1}(\sigma^m(x,u) \sigma^m(x,u)\tp -a^{m}(x)) 
(\underline{\sigma}^{m}(x))\tp)^{-1}$ in which zero columns are eliminated
to obtain a rectangular matrix $\Sigma^m(x,u)$
of size $d\times \ell$ when the rank of the initial matrix is equal to $\ell<d$.

Denote also by $\gu^m(x,u)$ the $d$-dimensional vector such that
\begin{equation}\label{assumpt-gu}
 f^m(x,u)-\underline{f}^{m}(x)=\underline{\sigma}^{m}(x) \gu^m(x,u)\enspace .
\end{equation}

Define
\begin{subequations}
\begin{align}
\G^{m}_1(x, p,\gu)&:= (\underline{\sigma}^{m}(x) \gu ) \cdot p \\
\G^{m}_2(x, \Gamma,\Sigma)&:= \frac{1}{2} \tr\left(\underline{\sigma}^{m}(x) \Sigma \Sigma\tp \underline{\sigma}^{m}(x)\tp\Gamma\right)
\end{align}
\end{subequations}
so that 
\[ \G^{m,u}(x, r,p,\Gamma)
= \ell^m(x,u)-\delta^m(x,u) r + \G^{m}_1(x, p, \gu^m(x,u) )+\G^{m}_2(x, \Gamma,
\Sigma^m(x,u) )\enspace .\]

Applying Theorems~\ref{theo-er3} and~\ref{theo-firstorder} and 
Lemma~\ref{lemma-zeroorder}, we deduce the following result
which shows the consistency of the scheme~\eqref{discHJB},
together with estimates that are necessary to apply the 
results of Barles and Jakobsen in~\cite{jakobsen07}.

\begin{theorem}\label{cor-const}
Let $\underline{\sigma}^m$, $\underline{f}^{m}$,
$\hat{X}^m$ and $\Li^m$ be as above.
Let us consider the following ``discretization'' operators from the set
of functions from $\bT_h\times \R^d$ to $\R$ to the set
of functions from  $\T_h\times \R^d$ to $\R$ :
\begin{align*}
\DD_{t,h,m}^0(\phi)(t,x)&:=
\EE\left[\phi(t+h, \hat{X}^m(t+h)) \mid \hat{X}(t)=x \right] \\
\DD_{t,h,m,\gu}^1(r,\phi)(t,x)&:=
\EE\left[(\phi(t+h, \hat{X}^m(t+h))-r)\Po^1_{\gu}(h^{-1}(W_{t+h}-W_t)) 
\mid \hat{X}(t)=x \right] \\
\DD_{t,h,m,\Sigma,k}^2(\phi)(t,x)&:=
h^{-1}\EE\left[\phi(t+h,\hat{X}^m(t+h)) \Po^2_{\Sigma,k}(h^{-1/2}(W_{t+h}-W_t))
 \mid \hat{X}^m(t)=x\right]\enspace ,
\end{align*}
with  $\Po^1_{\gu}$ and $\Po^2_{\Sigma,k}$ as in~\eqref{poly1}
and~\eqref{defpoly2k} respectively.

Then, consider the following discretization of~\eqref{HJB}:
\begin{equation}\label{discHJB}
 \K(h,t,x,v(t,x),v)=0,\quad (t,x)\in \T_h\times \R^d\enspace ,
\end{equation}
where $v$ is a map  $\bT_h\times \R^d$ to $\R$, and $\K$ is 
defined by: 
\begin{align*}
&\K(h,t,x,r,\phi)=-
 \max_{m\in\ M,\; u\in \U} \Big\{ h^{-1} (\DD_{t,h,m}^{0}(\phi)(t,x)-r)\\
&\quad +  \ell^m(x,u)- \delta^m(x,u) r
+\DD_{t,h,m,\gu^m(x,u)}^1(r,\phi)(t,x)
+  \DD_{t,h,m,\Sigma^m(x,u),k}^2(\phi)(t,x)\Big\}
\enspace .
\end{align*}
Assume that $\underline{\sigma}^m$, $\underline{f}^{m}$, $\gu^m$ and 
$\Sigma^m$ are bounded maps (in $x$ and $u$). Then, there exists $K$ 
depending on these bounds, such that,
for any $0<\epsilon\leq 1$, $\tilde{K}$ and $v\in\C^{\infty}\bo $
satisfying 
\[ |\partial^{p}_t D^{q} v| \leq \tilde{K} \epsilon^{1-2p-q}\quad
\text{for all}\;  p,q\in \N\enspace ,\]
we have,
\[ |\K(h,t,x,v(t,x),v)
+\frac{\partial v}{\partial t}(t,x)
+\Ha(x, v(t,x), Dv(t,x), D^2v(t,x))|\leq E(\tilde{K},h,\epsilon) \enspace ,\]
 for all $t\in\T_h$ and $x\in \R^d$, with
\[  E(\tilde{K},h,\epsilon)=K\tilde{K} \left( h\epsilon^{-3} (1+\sqrt{h})^4
(1+\sqrt{h}\epsilon^{-1})^4+ \sqrt{h}\epsilon^{-1} (1+\sqrt{h})^2
(1+\sqrt{h}\epsilon^{-1})^2\right)\enspace .\]
\end{theorem}

\begin{lemma}\label{remarkT}
If $\delta^m\geq 0$, or if $\delta^m$ is lower bounded and $h$
is small enough, 
the discretized equation~\eqref{discHJB} can be rewritten as
the solution of the iterative equation~\eqref{scheme} with
$T_{t,h}$ defined by:
\begin{align}\label{defnewTh}
T_{t,h}(\phi)(x)=&   \max_{m\in\ M,\; u\in \U} \frac{T_{t,h,m,u}^N(\phi)(x)}
{T_{t,h,m,u}^D(x)}\enspace,
\end{align}
with
\begin{align*}
T_{t,h,m,u}^N(\phi)(x)=&
\DD_{t,h,m}^{0}(\phi)(t,x) +h \big\{
\ell^m(x,u)\\
&\quad  + \DD_{t,h,m,\gu^m(x,u)}^1(0,\phi)(t,x)
+  \DD_{t,h,m,\Sigma^m(x,u),k}^2(\phi)(t,x)\big\}\big\}\\
T_{t,h,m,u}^D(x)=&
1+h \delta^m(x,u)+h \EE\left[\Po^1_{\gu^m(x,u)}(h^{-1}(W_{t+h}-W_t)) \right]
\enspace .
\end{align*}
\end{lemma}

Note that $T_{t,h,m,u}^D(x)=1+O(\sqrt{h})$ when $\delta^m$ and
$\gu^m$ are upper bounded.

\begin{remark}\label{remcompare}
When $\delta^m(x,u)$ and $\gu^m(x,u)$ are zero, the above operator $T_{t,h}$ 
coincides with the operator proposed in~\cite{fodjo2}, 
which corresponds to
\begin{subequations}
\begin{align}
T_{t,h}(\phi)(x)=&   \max_{m\in\ M,\; u\in \U} {T_{t,h,m,u}(\phi)(x)}\\
T_{t,h,m,u}(\phi)(x)=&
\DD_{t,h,m}^{0}(\phi)(t,x) (1-\delta^m(x,u)h)
+h \big\{ \ell^m(x,u) \\
&\quad  + \tilde{\DD}_{t,h,m,\gu^m(x,u)}^1(\phi)(t,x)
+  \DD_{t,h,m,\Sigma^m(x,u),k}^2(\phi)(t,x)\big\}\big\}
\enspace , 
\end{align}
\label{Thfodjo2}
\end{subequations}
with
\[ \tilde{\DD}_{t,h,m,\gu}^1(\phi)(t,x):=
\EE\left[\phi(t+h, \hat{X}^m(t+h))\gu \cdot (h^{-1}(W_{t+h}-W_t))
\mid \hat{X}(t)=x \right] \enspace .\]
When $k=0$, and $\Li^m=\Li$ does not
depend on $m$, the former operator coincides 
with the operator~\eqref{def-th-m} 
proposed in~\cite{touzi2011}, see~\cite{fodjo2}.
Note that when $\delta^m(x,u)\neq 0$,
one need to replace $-\delta^m(x,u) r$
by  $- \delta^m(x,u) \DD_{t,h,m}^{0}(\phi)(t,x)$ in the expression of $\K$
in order to recover the operators of~\cite{touzi2011} and~\cite{fodjo2}.

When the sign of $\delta^m$ is not fixed or $\delta^m$ is not lower bounded,
one can replace   $-\delta^m(x,u) r$ by  
\[-\delta^m(x,u)_+ r+\delta^m(x,u)_- \DD_{t,h,m}^{0}(\phi)(t,x)\]
in the expression of $\K$ so that in all cases,
the discretized equation~\eqref{discHJB} can be rewritten as
the solution of the iterative equation~\eqref{scheme} with
$T_{t,h}$ defined by~\eqref{defnewTh} and 
\begin{align*}
T_{t,h,m,u}^N(\phi)(x)=&
\DD_{t,h,m}^{0}(\phi)(t,x) +h \big\{
\ell^m(x,u)+\delta^m(x,u)_- \DD_{t,h,m}^{0}(\phi)(t,x)\\
&\quad  + \DD_{t,h,m,\gu^m(x,u)}^1(0,\phi)(t,x)
+  \DD_{t,h,m,\Sigma^m(x,u),k}^2(\phi)(t,x)\big\}\big\}\\
T_{t,h,m,u}^D(x)=&
1+h \delta^m(x,u)_++h \EE\left[\Po^1_{\gu^m(x,u)}(h^{-1}(W_{t+h}-W_t)) \right]
\enspace .
\end{align*}
\end{remark}

In~\cite[Theorem 3.3]{fodjo2}, we proved that 
the operator $T_{t,h}$ is monotone for $h$ small enough
over the set of bounded continuous functions $\R^d\to\R$,
under the assumption that $\bar{a}< 4k+2$ with $\bar{a}$  an
upper bound of $\tr(\Sigma^m(x,u)\Sigma^m(x,u)\tp)$ (for all
$x$ and $u$) and that $\delta^m$ is upper bounded, 
and that there exists a bounded map $\tilde{g}^m$ such that
$\gu^m(x,u)= \Sigma^m(x,u)\tilde{g}^m(x,u)$.
This was already a generalization of~\cite[Lemma 3.12]{touzi2011},
since the latter corresponds to the case where $k=0$.
Here, we shall only need that $\gu^m$ is bounded.
This will allows to apply the result to degenerate matrices 
$\Sigma^m(x,u)\Sigma^m(x,u)\tp$.
Also $\delta^m$ need not to be upper bounded at this point
because the expression of $\K$ uses $-\delta^m(x,u) r$ instead of
$- \delta^m(x,u) \DD_{t,h,m}^{0}(\phi)(t,x)$.
\begin{theorem}\label{theo-monotone}
Let $\K$ be as in Theorem~\ref{cor-const}.
Assume that the map $\tr(\Sigma^m(x,u)\Sigma^m(x,u)\tp)$ is upper bounded 
in $x$ and $u$ and let $\bar{a}$ be  an upper bound.
Assume also that %
$\delta^m$ is lower bounded.
Then, for $k$ such that $\bar{a}\leq  4k+2$,
 $\K$ is monotone in the sense of~\cite{jakobsen07}.
Also, there exists $h_0$ such that the 
operator $T_{t,h}$ of Lemma~\ref{remarkT} is monotone for $h\leq h_0$
over the set of bounded continuous functions $\R^d\to\R$. 
\end{theorem}
\begin{proof}
Adapting the definition of monotonicity of~\cite[{\bf (S1)}]{jakobsen07} to our
setting (backward equations and a time discretization only), we need to prove
that there exists $\lambda, \mu \ge 0$, $h_0 > 0$ such that if $h\leq h_0$, 
$v,v'$ are bounded continuous functions from $\bT_h\times \R^d$ to $\R$
 such that 
$v \leq v'$ and $\psi(t)=e^{\mu (T-t)} (a+b(T-t))+c$ with $a,b,c \ge 0$, then :
\begin{equation}\label{monotone-barles}
\K(h,t,x,r+\psi(t),v+\psi) \geq \K(h,t,x,r,v') + b/2-\lambda c \text{ in }\T_h\times \R^d\enspace .\end{equation}
Let us first show the inequality for $\psi=0$.
Using the notations of Lemma~\ref{remarkT}, we have
\begin{equation}\label{writeK}
\K(h,t,x,r,\phi)=-\max_{m\in\ M,\; u\in \U}  
h^{-1} \left( T_{t,h,m,u}^N(\phi(t+h,\cdot))(x) - T_{t,h,m,u}^D(x) r \right)\enspace.
\end{equation}
Also
\begin{align*}
 T_{t,h,m,u}^N(\phi)(x) =& h \ell^m(x,u)\\
&+
\EE\left[\phi(\hat{X}^m(t+h)) \Po^{h,m,u,x}(h^{-1/2}(W_{t+h}-W_t))
 \mid \hat{X}^m(t)=x\right]\enspace ,\end{align*}
where 
\begin{align*}
 \Po^{h,m,u,x}(w)=&1+h \Po^1_{\gu^m(x,u)}(h^{-1/2}w) +\Po^2_{\Sigma^m(x,u),k}(w)\enspace .
\end{align*}
Since $\Po^1_{\gu}\geq 0$ for all $\gu$ and 
$\Po^2_{\Sigma}\geq -\frac{\tr(\Sigma\Sigma\tp)}{4k+2}$ for all $\Sigma$,
we get that $\Po^{h,m,u,x}(w)\geq 1-\frac{\bar{a}}{4k+2}$.
Assume now that $\bar{a}\leq  4k+2$.
Then, $\Po^{h,m,u,x}(w)\geq 0$, so
 if $v \leq v'$, then $T_{t,h,m,u}^N(v)\leq T_{t,h,m,u}^N(v')$ and
$\K(h,t,x,r,v) \geq \K(h,t,x,r,v')$.

To show~\eqref{monotone-barles}, it is now sufficient to show 
the same inequality for $v=v'$.
We have
\begin{align*}
& \K(h,t,x,r+\psi(t),v+\psi) -\K(h,t,x,r,v)
\geq -\max_{m\in\ M,\; u\in \U} \Big\{ h^{-1} (\psi(t+h)-\psi(t))\\
&\quad - \delta^m(x,u) \psi(t)
+(\psi(t+h)-\psi(t))\EE[ \Po^1_{\gu^m(x,u)}(h^{-1}(W_{t+h}-W_t))]
\Big\}\enspace .
\end{align*}
Let us take for $\lambda$ an upper bound of $-\delta^m$.
From $\psi(t+h)-\psi(t)\leq 0$, and  $\Po^1_{\gu}\geq 0$ for all $\gu$,
we deduce
\begin{align*}
& \K(h,t,x,r+\psi(t),v+\psi) -\K(h,t,x,r,v) \\
& \quad \geq - h^{-1} (\psi(t+h)-\psi(t))-\lambda  \psi(t) \\
& \quad = b e^{\mu (T-t-h)} +  e^{\mu (T-t)} (\frac{1-e^{-\mu h}}{h}-\lambda) (a+b(T-t))-\lambda c\\
&\quad \geq  b-\lambda c\enspace ,
\end{align*}
if $1-e^{-\mu h}\geq \lambda h$.
Taking $\mu>\lambda$, there exists $h_0$ such that $1-e^{-\mu h}\geq \lambda h$
for all $h\leq h_0$, leading to the previous inequality and so 
to~\eqref{monotone-barles} for $v=v'$. This shows the 
that $\K$ is monotone in the sense of~\cite{jakobsen07}.

Since $\Po^1_{\gu}\geq 0$ for all $\gu$,
and $\lambda\geq -\delta^m$,  we get also that
$T_{t,h,m,u}^D(x)\geq 1-\lambda h$ and so $T_{t,h,m,u}^D(x)>0$ for 
$h\leq h_0$ if $h_0 < 1/\lambda$. 
Since we already proved that $T_{t,h,m,u}^N$ is monotone,
for all $m,u$, we obtain that the operator $T_{t,h}$ of Lemma~\ref{remarkT}
is well defined and monotone for $h\leq h_0$ 
over the set of bounded continuous functions $\R^d\to\R$.
\end{proof}

We shall say that an operator $T$ between any sets $\F$ and $\F'$
of partially ordered sets of
real valued functions, which are stable by the addition of a constant function
(identified to a real number),
is \NEW{additively $\alpha$-subhomogeneous} if 
\begin{equation}\label{def-hom}
\lambda\in\R,\lambda \geq 0,\; \phi\in \F\; 
 \implies T(\phi+\lambda)\leq T(\phi)+\alpha \lambda 
\enspace .\end{equation}

\begin{lemma}\label{lem-hom}
Assume that $\delta^m$ is lower bounded in $x$ and $u$
and let $T_{t,h}$ be as in Lemma~\ref{remarkT}.
Then, there exists $h_0>0$ such that for $h\leq h_0$,
$T_{t,h}$ is additively $\alpha_h$-subhomogeneous
over the set of bounded continuous functions $\R^d\to\R$,
for some constant $\alpha_h=1+Ch$ with $C\geq 0$.
\end{lemma}
\begin{proof}
If $\lambda$ is an upper bound of $-\delta^m$, take
$C=2\lambda$ and $h_0$ such that $1-\lambda h_0\geq 1/2$.
\end{proof}
With the monotonicity, the $\alpha_h$-subhomogeneity implies the 
$\alpha_h$-Lipschitz continuity of the operator, which allows one to show 
easily the stability as follows, see~\cite[Corollary~3.5]{fodjo2} for the proof.
\begin{corollary}\label{cor-stability}
Let the assumptions and conclusions of Theorems~\ref{cor-const} 
and~\ref{theo-monotone} hold
and assume also that $\psi$ and $\ell^m$ are bounded.
Then, there exists a unique function $v^h$ on $\T_h\times \R^d$ 
satisfying~\eqref{discHJB} or equivalently~\eqref{scheme}
with $T_{t,h}$ as in  Lemma~\ref{remarkT}
and $v^h(T,x)=\psi(x)$ for all $x\in \R^d$.
Moreover $v^h$ is bounded (independently of $h$).
\end{corollary}
Note that  the assumptions can be summerized in ``all the maps
$\psi$,  $\ell^m$, 
$\underline{\sigma}^m$, $\underline{f}^{m}$, $\gu^m$ and 
$\Sigma^m$ are bounded, and the map $\delta^m$ 
is lower bounded (which is equivalent to say that 
the map $e^{-\delta^m}$ is bounded).
This implies that $f^m$ and $\sigma^m (\sigma^m)\tp$ are bounded,
and, if $\sigma^m$ is symmetric then $\sigma^m$ is also bounded, 
but we do not need this directly.

\begin{corollary}\label{cor-stabilityplus}
Let the assumptions and conclusions of Corollary~\ref{cor-stability} hold.
Assume also that all the maps 
$\psi$,  %
$\delta^m,\ell^m$, $\underline{\sigma}^m$, $\underline{f}^{m}$, $\gu^m$ and 
$\Sigma^m$ are continuous 
with respect to $x\in \R^d$, uniformely in $x$ and $u\in \U$.
Then the unique solution $v^h$ of~\eqref{discHJB},
with the initial condition $v^h(T,x)=\psi(x)$ for all $x\in \R^d$,
is uniformely continuous on $\bT_h\times \R^d$.
\end{corollary}
\begin{proof}
Since $\bT_h$ is finite, we just need to show that $v^h(t,\cdot)$
is uniformely continuous on $\R^d$ for all $t\in \bT_h$.
Since $v^h(T,\cdot)=\psi$ which is already bounded and uniformely continuous 
on $\R^d$, we only need to show that the operator
$T_{t,h}$ of Lemma~\ref{remarkT} sends the set of bounded and
uniformely continuous functions on $\R^d$ to itself. 
From the proof of Corollary~\ref{cor-stability}, it sends 
bounded functions to bounded functions.
So, it is sufficient to show that $T_{t,h,m,u}^D$ is
uniformely continuous, uniformely in $u\in \U$ and that 
$T_{t,h,m,u}^N$ sends bounded uniformely continuous functions on $\R^d$ to 
functions that are uniformely continuous in $x$ uniformely in $u\in \U$.
The first property is due to the uniform continuity of $\delta^m$ and
$\gu^m$ uniformely in $u\in \U$.
For the second one, one uses that if $\hat{X}^m(t)=x$, then
$\hat{X}^m(t+h)= x + \underline{f}^m(x)h+ \underline{\sigma}^m(x) 
(W_{t+h}-W_t)$ which is uniformely continuous in $x$, for all given 
values of $W_{t+h}-W_t$, since $\underline{\sigma}^m$ and $\underline{f}^{m}$
are uniformely continuous in $x$.
Hence, when $\phi$ is bounded and uniformely continuous
with respect to $x$, then $\phi(\hat{X}^m(t+h))$ is bounded and 
uniformely continuous with respect to $x$, for all given 
values of $W_{t+h}-W_t$. Since all moments of $W_{t+h}-W_t$ are finite and
the maps $\ell^m$, $\gu^m$ and $\Sigma^m$ are uniformely continuous 
with respect to $x\in \R^d$, uniformely in $u\in \U$,
we deduce that $T_{t,h,m,u}^N(\phi)$ is 
uniformely continuous in $x$, uniformely in $u\in \U$.
\end{proof}
The previous result shows that the map $v^h$ can be extended in a
continuous function over  $[0,T]\times \R^d$.
Then, the convergence of the scheme can be obtained as in~\cite{fodjo2}
by applying the theorem
of Barles and Souganidis~\cite{barles90}:

\begin{corollary}
Let the assumptions of Corollary~\ref{cor-stabilityplus} hold.
Assume also that~\eqref{HJB} has a strong uniqueness property 
for viscosity solutions
and let $v$ be its unique viscosity solution.
Let $v^h$ be the unique solution of~\eqref{discHJB},
with the initial condition $v^h(T,x)=\psi(x)$ for all $x\in \R^d$.
Let us  extend $v^h$ on $[0,T]\times \R^d$ as a continuous 
and piecewise linear function with respect to $t$.
Then, when $h\to 0^+$, $v^h$ converges to $v$ locally 
uniformely in $t\in [0,T]$ and $x\in \R^d$.
\end{corollary}

To apply the theorem of Barles and Jakobsen~\cite{jakobsen07}, we 
also need the following regulatity result 
(corresponding to {\bf (S2)} in~\cite{jakobsen07}) which is 
comparable to the previous one.

\begin{lemma}\label{lemma-regular}
Let the assumptions of Corollary~\ref{cor-stabilityplus} hold.
Assume also that $\delta^m$ is bounded.
Then, for all 
continuous and bounded function $v$ on $\bT_h\times \R^d$, the function
$(t,x) \mapsto \K(h,t,x,v(t,x),v)$ is bounded and continuous in 
$\T_h\times \R^d$.
Moreover, the function $r \mapsto \K(h,t,x,r,v)$ is uniformly continuous for bounded $r$, uniformly in $(t,x) \in \T_h\times \R^d$.
\end{lemma}
\begin{proof}
Using the arguments of the  proof of Corollary~\ref{cor-stabilityplus}
and the rewritting of $\K$ in~\eqref{writeK}, one gets that 
$x \mapsto \K(h,t,x,r,v)$ is uniformely continuous in $x$, uniformely
in $r$ bounded. Also since  $\delta^m$ and $\gu^m$ are bounded,
then $T_{t,h,m,u}^D$ is bounded, so $(x,r) \mapsto \K(h,t,x,r,v)$ is 
uniformely continuous in $x\in \R^d$ and $r$ in a bounded set of $\R$.
This shows in particular that $r \mapsto \K(h,t,x,r,v)$ is 
uniformely continuous in $r$ bounded, uniformely in
$x\in \R^d$. Also, since $v$ is bounded and uniformely continuous,
this implies that $x \mapsto \K(h,t,x,v(t,x),v)$ is bounded and continuous in 
$\R^d$.
Since $\T_h$ is a finite set, the assertions of the lemma follow.
\end{proof}

We also need the following assumptions which correspond to the
assumptions with same names in~\cite{jakobsen07}.

\catcode`\@=11
\def\refcounter#1{\protected@edef\@currentlabel
       {\csname the#1\endcsname}%
}
\catcode`\@=12
\newcounter{assume}
\def\theassume{{\bf (A\arabic{assume})}}
\def\myitem{\refstepcounter{assume}\item[\theassume] \hskip  2.5ex}

For a function $v$ defined on $\R^d$, $|v|_0$ and $|v|_1$
will denote respectively the norm on the
space of bounded functions (that is the sup-norm) 
and the norm on the space of bounded
Lipschitz continuous functions on $\R^d$ (that is the sup-norm plus the
minimal Lipschitz constant).
More generally, for a function defined on $Q=[0,T]\times \R^d$, $|v|_0$ 
will denote the sup-norm,  while $|v|_1$ will denote a norm
on the space of bounded functions that are Lipschitz continuous
with respect to $x$ and
$1/2$-H\"older continuous with respect to $t$:
$$|v|_0 = \underset{(t,x) \in Q}{\sup}|v(t,x)|\enspace, \quad|v|_1 =
|v|_0 + \underset{\substack{(t,x)\in Q \\ (t',x')\in Q' \\ (t,x) \neq (t',x')}}{\sup}\frac{|v(t',x')-v(t,x)|}{(t'-t)^{1/2}+|x'-x|}\enspace .$$

\begin{itemize} %
  \myitem \label{A1} There exists a constant $K>0$, such that
\[ |\phi|_1\leq K\]
for $\phi=\psi$ and for all the maps $\phi=h(\cdot,u)$ with
$h$ beeing any coordinate of the maps $f^m,\sigma^m,\delta^m,\ell^m$,
and any $m\in\M$ and $u\in \U$.
\end{itemize}
\begin{itemize}
\myitem \label{A2} For every $\delta >0$, there is a finite subset 
$\U_F$ of $\U$ such that for any $u \in \U$, there exists $u_F\in \U_F$
such that
$$ |h(\cdot,u)-h(\cdot,u_F)|_0\leq \delta$$
for all the maps $h$ beeing any coordinate of the maps $f^m,\sigma^m,\delta^m,\ell^m$, and any $m\in\M$.
\end{itemize}

Applying~\cite[Theorem 3.1]{jakobsen07}, we obtain 
 the following estimations which are of the
same order as the ones obtained for usual
explicit finite difference schemes with $\Delta x$ in the order of
$\sqrt{h}$~\cite{jakobsen07} 
or for the scheme of~\cite{touzi2011}.

\begin{corollary}
Let the assumptions of Corollary~\ref{cor-stabilityplus} hold.
Assume also \ref{A1} and \ref{A2}.
Let $v$ be the unique viscosity solution of~\eqref{HJB} and
$v^h$ be the unique solution of~\eqref{discHJB},
with the initial condition $v^h(T,x)=\psi(x)$ for all $x\in \R^d$.
Then,  there exists $C_1,C_2$ depending on $|v|_1$ such that,
for all $(t,x)\in \bT_h\times \R^d$, we have
\[  -C_1 h^{1/10} \leq (v^h-v)(t,x)\leq C_2 h^{1/4} \enspace .\]
\end{corollary}

\section{The probabilistic max-plus method}\label{sec-maxplus}

In~\cite{touzi2011}, the solution $v^h$ 
of the time discretization~\eqref{scheme} 
of the partial differential equation~\eqref{HJB} is obtained by using the
following method which can be compared to a space discretization.
The conditional expectations in~\eqref{defDcal} are approximated
by any probabilistic method such as a regression estimator:
after a simulation of the processes $W_t$ and $\hat{X}(t)$,
one apply at each time $t\in\T_h$ a regression estimation to find the value
of $\DD_{t,h}^i(v^h(t+h,\cdot))$ at the points $\hat{X}(t)$ by using
the values of  $v^h(t+h,\hat{X}(t+h))$ and $W_{t+h}-W_t$.
The regression can be done over a finite dimensional linear space
approximating the space of bounded Lipschitz continuous functions,
for instance  the linear space of functions that are polynomial 
with a certain degree on some ``finite elements''.
Hence, the value function $v^h(t,\cdot)$ is obtained by an estimation
of it at the simulated points  $\hat{X}(t)$.
This method can also be used for the scheme~\eqref{scheme} 
obtained in the previous section, since the new one also involve 
conditional expectations.

In the probabilistic max-plus method proposed in~\cite{fodjo1} and 
used in~\cite{fodjo2}, the aim was to replace the (large) finite dimensional
linear space of functions used in the regression estimations by the
max-plus linear space of  max-plus linear combinations of functions 
that belong to a small dimensional linear space 
(such as the space of quadratic forms). The idea is that stochastic control
problems involve at the same time an expectation which is a linear operation
and a maximization which is a max-plus linear operation.
Note that a direct regression estimation on such 
a non linear space is difficult.
We rather used the distributivity property
of monotone operators over suprema operations,
recalled in Theorem~\ref{main-theo} below, a property which generalizes the one
shown in Theorem 3.1 of McEneaney, Kaise and Han~\cite{mceneaney2011}.
This allowed us to 
reduce the regression estimations to the small dimensional linear space 
of quadratic forms.

The algorithm of~\cite{fodjo1} was based on the scheme
of~\cite{touzi2011}, that is~\eqref{scheme} with $T_{t,h}$ 
as in~\eqref{def-th-m}.
The one of~\cite{fodjo2} was based on~\eqref{scheme} with $T_{t,h}$ 
involving the discretization of second order terms as
in Theorem~\ref{theo-er3}
with $k$ large enough in such a way that the scheme is monotone, that 
is the scheme of Theorem~\ref{cor-const}
 but with a discretization of zero and first
order terms as in~\eqref{def-th-m}, see Remark~\ref{remcompare}.
Here, we shall explain how the algorithm can be adapted to the
case of the discretization of Theorem~\ref{cor-const}.

In the sequel, we denote $\W=\R^d$ and $\DD$ 
the set of measurable functions from $\W$ to $\R$ 
with at most some given growth or growth rate (for instance with at most
exponential growth rate), assuming that it contains the constant functions.

\begin{theorem}[\protect{\cite[Theorem 4]{fodjo1}}]\label{main-theo}
Let $G$ be a monotone additively $\alpha$-subhomogeneous operator
from $\DD$ to $\R$, for some constant $\alpha>0$.
Let $(Z,\A)$ be a measurable space, and let $\W$ be endowed with its
Borel $\sigma$-algebra. 
Let $\phi:\W \times Z\to\R$ be a measurable map such that for all $z \in Z$, $\phi(\cdot,z)$ is %
continuous and belongs to $\DD$. 
Let $v\in\DD$ be such that
$v(W)=\sup_{z\in Z} \phi(W,z)$. Assume that $v$ is continuous and 
bounded. %
Then,
\[ G(v)= \sup_{\bar{z}\in\overline{Z}}
G(\bar{\phi}^{\bar{z}})\] 
where $\bar{\phi}^{\bar{z}} :  \W\to \R,\; W \mapsto \phi(W,\bar{z}(W))$,
and 
\begin{align*}
\overline{Z}=&\{\bar{z}\, : \W \to Z,\; \text{measurable
and such that}\; \bar{\phi}^{\bar{z}} \in\DD\}.
\end{align*}
\end{theorem}

To explain the algorithm, assume
that the final reward $\psi$ of the control problem
can be written as the supremum of a finite number 
of concave quadratic forms.
Denote $\Q_d=\Sy_d^-\times \R^d\times \R$, where
$\Sy_d^-$ is the set of nonpositive symmetric $d\times d$
matrices, and let 
\begin{equation}\label{paramquad}
 q(x,z):= \frac{1}{2} x\tp Q x + b\cdot x +c, \quad
\text{with}\;\; z=(Q,b,c)\in \Q_d\enspace, \end{equation}
be the quadratic form with parameter $z$ applied to the vector $x\in\R^d$.
Then for $g_T=q$, we have
\[ v^h(T,x)=\psi(x)=\sup_{z\in Z_T} g_T(x,z)\]
where $Z_T$ is a finite subset of $\Q_d$.

The application of the operator  $T_{t,h}$ of Lemma~\ref{remarkT}
 to a (continuous)
function $\phi : \R^d \to \R, x\mapsto \phi(x) $ can be written, for each 
$x\in \R^d$, as
\begin{subequations}\label{def-T-G}
\begin{align}
T_{t,h}(\phi)(x)&=\max_{m\in \M}G_{t,h,x}^m(\tilde{\phi}^{m}_{t,h,x})\enspace , 
\end{align}
where 
\begin{align}
&S_{t,h}^m: \R^d\times \W \to \R^d, \; (x,W) \mapsto
S_{t,h}^m(x,W)= x + \underline{f}^m(x)h+ \underline{\sigma}^m(x) W
\enspace ,\label{def-Sm}\\
&\tilde{\phi}^{m}_{t,h,x}  =\phi(S_{t,h}^m(x,\cdot))\in \DD
\quad \text{if}\; \phi\in \DD\enspace ,
\label{def-tilde}
\end{align}
\end{subequations}
and $G_{t,h,x}^m$ is the operator from $\DD$ to $\R$ given by
\begin{eqnarray}
G_{t,h,x}^m(\tilde{\phi})&
=& \max_{u\in \U} \frac{G_{t,h,x,m,u}^N(\tilde{\phi})}{T_{t,h,m,u}^D(x)}\enspace,
\label{defG}\end{eqnarray}
with
\begin{eqnarray}
G_{t,h,x,m,u}^N(\tilde{\phi}) &=& D_{t,h}^{0}(\tilde{\phi})
+h  \big\{
\ell^m(x,u) +D_{t,h,\gu^m(x,u)}^1(\tilde{\phi}) +D_{t,h,\Sigma^m(x,u),k}^2(\tilde{\phi}) 
\big\}
\enspace ,
\label{defGN}\end{eqnarray}
\begin{align*}
&D_{t,h}^0(\tilde{\phi})=
\EE(\tilde{\phi}(W_{t+h}-W_t))\enspace ,\\
&D_{t,h,\gu}^1(\tilde{\phi})=
\EE(\tilde{\phi}(W_{t+h}-W_t)  \Po^1_{\gu}(h^{-1}(W_{t+h}-W_t))\enspace ,\\
& D_{t,h,\Sigma,k}^2(\tilde{\phi})(x):=
h^{-1}\EE\left[\tilde{\phi}(W_{t+h}-W_t) \Po^2_{\Sigma,k}(h^{-1/2}(W_{t+h}-W_t))\right]
\enspace ,
\end{align*}
$\gu^{m}(x,u)$ and $\Sigma^m(x,u)$, as in Section~\ref{sec-monotone}, 
and $\Po^1_{\gu}$ and $\Po^2_{\Sigma,k}$ as in~\eqref{poly1}
and~\eqref{defpoly2k} respectively.
Indeed, the Euler discretization  $\hat{X}^m$ 
of the diffusion with generator $\Li^m$ satisfies
\begin{equation}\label{def-hatxm}
\hat{X}^m(t+h)= S_{t,h}^m(\hat{X}^m(t), W_{t+h}-W_t)\enspace .
\end{equation}

Using the same arguments as for Theorem~\ref{theo-monotone} and 
Lemma~\ref{lem-hom},  one can obtain the stronger property
that for $h\leq h_0$, all the operators $G_{t,h,x}^m$ belong to the
class of monotone additively $\alpha_h$-subhomogeneous operators
from $\DD$ to $\R$.
This allows us to apply Theorem~\ref{main-theo}.
In~\cite{fodjo1}, we shown the following result.
\begin{theorem}[\protect{\cite[Theorem 2]{fodjo1}, compare 
with~\cite[Theorem 5.1]{mceneaney2011}}]\label{th-dist}
Consider the control problem of Section~\ref{sec-int}.
Assume that, for each $m\in\M$,
$\delta^m$ and $\sigma^m$ are constant, $\sigma^m$ is nonsingular,
$f^m$ is affine with respect to $(x,u)$, 
$\ell^m$ is concave quadratic with respect to $(x,u)$, and that $\psi$
is the supremum of a finite number of concave quadratic forms.
Consider the scheme~\eqref{scheme}, 
with $T_{t,h}$ as in~\eqref{Thfodjo2}, 
$\underline{\sigma}^m$ constant and nonsingular,
$\Sigma^m$ constant and nonsingular and $\underline{f}^m$ affine.
Assume that the %
operators $G_{t,h,x}^m$ belong to the
class of monotone additively ${\alpha}_h$-subhomogeneous operators
from $\DD$ to $\R$, for some constant ${\alpha}_h=1+{C}h$ with 
${C}\geq 0$.
Assume also that the value function $v^h$ of~\eqref{scheme}
belongs to $\DD$ and is locally 
Lipschitz continuous with respect to $x$.
Then, for all $t\in\T_h$, there exists a set $Z_t$ 
and a  map $g_t:\R^d\times Z_t\to \R$ such that
for all $z\in Z_t$, $g_t(\cdot, z)$ is a concave quadratic form and
\begin{equation}\label{supquadt}
 v^h(t,x)=\sup_{z\in Z_t} g_t(x,z)\enspace .\end{equation}
Moreover, the sets $Z_t$ satisfy
$ Z_t= \M\times
\{\bar{z}_{t+h}:\W\to Z_{t+h}\mid \text{Borel measurable}\}$. %
\end{theorem}

Theorem~\ref{th-dist} uses Theorem~\ref{main-theo} together with
the property that, 
for each $m$, the operator $T_{t,h}^m$ such that
$T_{t,h}^m(\phi)(x)= G_{t,h,x}^m(\tilde{\phi}^{m}_{t,h,x})$,
with $G_{t,h,x}^m$ defined in the same way as in~\eqref{defG} but for 
$T_{t,h}$ as in~\eqref{Thfodjo2}, 
sends a random (concave) quadratic form that is upper bounded 
by a deterministic quadratic form into a (concave) quadratic form.
This means that if $\bar{z}$ is a measurable function from $\W$ to $\Q_d$
and $\tilde{q}_{x}$ denotes the measurable map
$\W\to\R,\; W \mapsto q(S_{t,h}^m(x,W), \bar{z}(W) )$, 
with $q$ as in~\eqref{paramquad},
and if there exists $\bar{z}\in\Q_d$ such that
$\tilde{q}_{x}\leq q(x,\bar{z})$ for all $x\in \R^d$,
then the function $x \mapsto G_{t,h,x}^m (\tilde{q}_{x})$
 is a concave quadratic form, that is it can be written as $q(x,z)$ for some $z\in \Q_d$, see \cite[Lemma~3]{fodjo1}.

If we replace the operator $T_{t,h}$ of~\eqref{Thfodjo2}
by the one of Lemma~\ref{remarkT}, the previous
property does not hold because of the expressions
$\gu^+$ and $\gu^-$ and so one cannot deduce directly a result like
Theorem~\ref{th-dist}.
However, one can still obtain the following result:
\begin{lemma}\label{lem-imagequadprime}
 Let us consider the notations and assumptions
of Theorem~\ref{th-dist}, except that $T_{t,h}$ is replaced by
the operator of Lemma~\ref{remarkT}.
For each $m$, consider 
the operator $T_{t,h}^m$ such that
$T_{t,h}^m(\phi)(x)= G_{t,h,x}^m(\tilde{\phi}^{m}_{t,h,x})$ 
with $G_{t,h,x}^m$ as in~\eqref{defG}.
Let $\tilde{z}$ be a measurable function from $\W$ to $\Q_d$. 
Let $\tilde{q}^{m, \tilde{z}}_{t,h,x}$ be the map
$\W\to\R,\; W \mapsto q(S_{t,h}^m(x,W), \tilde{z}(W) )$, with $q$ as in~\eqref{paramquad}. Assume that there exists $\bar{z}\in\Q_d$ such that
$ q(x, \tilde{z}(W) )\leq q(x,\bar{z})$ for all $x\in\R^d$.
Then, the function $\bar{q}:x \mapsto G_{t,x,h}^m (\tilde{q}^{m, \tilde{z}}_{t,x,h})$ is 
upper bounded by a quadratic map and there exists $C>0$ and $z\in \Q_d$,
such that, for all $x\in \R^d$,
\[ q(x,z)\leq \bar{q}(x)\leq q(x,z)+ C h \sqrt{h} (\|x\|^2+1)^{3/2}\enspace .\]
\end{lemma}

This justify the application of the same algorithm as in~\cite{fodjo2},
that we recall  below for completeness for the operator of Lemma~\ref{remarkT}.
Recall that in the same spirit as in~\cite{touzi2011},
 we proposed in~\cite{fodjo1} and~\cite{fodjo2} to compute the expression of 
the maps $v^h(t,\cdot)$  by using simulations of the processes
$\hat{X}^{m}$. %
These simulations are not only used for regression estimations
of conditional expectations, which are computed there only in the
case of random quadratic forms, leading to quadratic forms,
but they are also used to
fix the ``discretization points'' $x$ at which the optimal
quadratic forms in the expression~\eqref{supquadt}
are computed. 

\noindent
\begin{algorithm}[\protect{\cite[Algorithm1]{fodjo2}}]\ \label{algo1}

\noindent
\emph{Input:} A constant $\eps$ giving the precision,
a time step $h$ and a horizon time $T$ such that $T/h$ is an integer,
a $3$-uple $N=(\Nzero,\NX,\NW)$ of integers giving the numbers of samples,
such that $\NX\leq \Nzero$, a subset $\Msmall\subset\M$ and a projection
map $\pi: \M\to\Msmall$.
A finite subset $Z_T$ of $\Q_d$ such that
$|\psi(x)-\max_{z\in Z_T} q(x,z)|\leq \eps$, for all $x\in\R^d$,
and $\card{Z_T}\leq \cardMs\times  \Nzero$.
The operators $T_{t,h}$ and $G_{t,x,h}^m$  as in {\rm (\ref{def-T-G}-\ref{defG})}
and the process $\hat{X}^{m}(t)$ satisfying~\eqref{def-hatxm}
for $t\in\T_h$, 
with $\Li^m$ (and thus  $\hat{X}^{m}$ and $G_{t,x,h}^m$) depending only
on $\pi(m)$.

\noindent
\emph{Output:}
The subsets $Z_t$ of $\Q_d$, for $t\in\T_h\cup\{T\}$, 
and the approximate value function
$v^{h,N}:(\T_h\cup\{T\})\times \R^d\to \R$.

\noindent
$\bullet$ \emph{Initialization:}
Let  $\hat{X}^{m}(0)=\hat{X}(0)$, for all $m\in \Msmall$,
where $\hat{X}(0)$ is random and independent of the Brownian process.
Consider a sample of $(\hat{X}(0),(W_{t+h}-W_t)_{t\in \T_h})$
of size $\Nzero$ indexed by $\omega\in \Omega_{\Nzero}:=\{1,\ldots, {\Nzero}\}$,
and denote, for each $t\in\T_h\cup\{T\}$,
$\omega \in \Omega_{\Nzero}$, and $m\in \Msmall$,
$\hat{X}^{m}(t,\omega)$ the value
of  $\hat{X}^{m}(t)$ induced by this sample.
Define $v^{h,N}(T,x)=\max_{z\in Z_T} q(x,z)$,
for $x\in \R^d$, with $q$ as in~\eqref{paramquad}.

\noindent
$\bullet$ For $t=T-h,T-2h,\ldots, 0$ apply the following 3 steps:

(1) Choose a random sampling $\omega_{i,1},\; i=1,\ldots, \NX$
among the elements of $\Omega_{\Nzero}$
and independently a random sampling  $\omega'_{1,j}\; j=1,\ldots, \NW$
among the elements of $\Omega_{\Nzero}$,
then take the product of samplings,
that is consider $\omega_{(i,j)}=\omega_{i,1}$ and $\omega'_{(i,j)}=\omega_{1,j}$
for all $i$ and $j$,
leading to $(\omega_{\ell},\omega'_{\ell})$ 
for $\ell\in \Omega_{\NXW}:=\{1,\ldots,\NX\}\times \{1,\ldots, \NW\}$.

Induce the sample $\hat{X}^{m}(t,\omega_{\ell})$ (resp.\ $(W_{t+h}-W_t)(\omega'_{\ell})$)
for $\ell\in \Omega_{\NXW}$
of $\hat{X}^{m}(t)$ with $m\in\Msmall$ (resp.\ $W_{t+h}-W_t$).
Denote by $\W^N_t\subset \W$
the set of $(W_{t+h}-W_t)(\omega'_{\ell})$ for $\ell \in \Omega_{\NXW}$.

(2) For each $\omega\in\Omega_{\Nzero}$ and $m\in \Msmall$,  
denote $x_t=\hat{X}^{m}(t,\omega)$ and
construct $z_t\in \Q_d$ depending on  $\omega$ and  $m$ as follows: 

(a) Choose 
$\bar{z}_{t+h}:\W^N_t\to Z_{t+h}\subset \Q_d$ such that,
for all $\ell\in \Omega_{\NXW}$, we have 
\begin{align*}
& v^{h,N}(t+h,S^m_{t,h}(x_t,(W_{t+h} -W_t)(\omega'_\ell)))\\
&\; = q\big(S^m_{t,h}(x_t,(W_{t+h} -W_t)(\omega'_\ell))
,\bar{z}_{t+h}((W_{t+h} -W_t)(\omega'_\ell))\big)\enspace .
\end{align*}
Extend $\bar{z}_{t+h}$ as a measurable map from $\W$ to $\Q_d$.
Let $\tilde{q}_{t,h,x}$ be the element of $\DD$ given by
$W \in \W  \mapsto q(S^m_{t,h}(x, W), \bar{z}_{t+h}(W) )$. 

(b) For each $\bar{m}\in\M$ such that $\pi(\bar{m} )=m$,
compute an approximation of  
$x \mapsto G_{t,h,x}^{\bar{m}} (\tilde{q}_{t,h,x})$ by a linear regression 
estimation on the set of quadratic forms using the sample
$(\hat{X}^{\bar{m}}(t,\omega_\ell), (W_{t+h} -W_t)(\omega'_\ell))$, with $\ell\in\Omega_{\NXW}$, and denote by $z_t^{\bar{m}}\in \Q_d$ the parameter of the resulting quadratic form.

r(c) Choose $z_t\in \Q_d$ optimal among the  $z_t^{\bar{m}}\in \Q_d$ at the point
$x_t$, that is such that
$q(x_t, z_t)= \max_{\pi(\bar{m})=m}q(x_t, z_t^{\bar{m}})$.

(3) Denote by $Z_t$ the set of all the $z_t\in \Q_d$ obtained in this way,
and define
\[ v^{h,N}(t,x)=\max_{z\in Z_t} q(x,z)\quad
\forall x\in \R^d \enspace .\]

\end{algorithm}
Recall that no computation is done at Step (3), which gives only a
formula to be able to compute the value function at each time step
and state $x$ by using the sets $Z_t$.

Contrarilly to what happened in~\cite{fodjo2}, the map 
$x \mapsto G_{t,h,x}^{\bar{m}} (\tilde{q}_{t,h,x})$ is not necessarily
a quadratic form, but for $x$ in a bounded set and $h$ small enough,
it can be approximated by a quadratic form, see Lemma~\ref{lem-imagequadprime}.
Then, the regression estimation over the set of quadratic forms
gives an approximation of order $O(h\sqrt{h})$ which add 
an error in $O(\sqrt{h})$ to the value function at time $0$.
In~\cite[Proposition~5]{fodjo1},
under suitable assumptions, we shown the convergence
$\lim_{\Nzero, \NXW\to\infty} v^{h,N}(t,x)= v^h(t,x)$.
Here, we may expect that 
$\limsup_{\Nzero, \NXW\to\infty} |v^{h,N}(t,x)-v^h(t,x)|\leq C \sqrt{h}$.
However a further study is needed to obtain a precise estimation
of the error depending on $\Nzero,\; \NXW$ and $h$.
\bibliographystyle{plain}

\bibliography{maxproba}

\end{document}